\documentclass[11pt]{amsart}

\usepackage{amsthm,amssymb,amsfonts,amsmath,enumerate,graphicx,tikz,listings,array,longtable}
\usetikzlibrary{shapes,arrows,chains} 

\usepackage[margin=1in]{geometry}
\usepackage{pgf}

\makeatletter
\newcommand{\addresseshere}{%
  \enddoc@text\let\enddoc@text\relax
}
\makeatother

\newcommand{\PP}{\mathrm{P}}
\newcommand{\B}{\mathbb{B}}
\newcommand{\Z}{\mathbb{Z}}
\newcommand{\N}{\mathbb{N}}
\newcommand{\R}{\mathbb{R}}

\newcommand{\fS}{\mathfrak{S}}
\newcommand{\Id}{\mathrm{Id}}

\newcommand{\cone}{\mathrm{cone}}

\newcommand{\conv}{\mathrm{conv}}

\newcommand{\height}{\mathrm{height}}
\newcommand{\im}{\mathrm{im}}
\newcommand{\HH}{\mathcal{H}}
\newcommand{\fpa}{\mathrm{FPA}}

\newcommand{\pn}{\mathrm{part}}

\newcommand{\rp}{\mathrm{relprime}}
\newcommand{\rpac}{\mathrm{rpac}}
\newcommand{\relprime}{\mathrm{relprime}}
\newcommand{\ptn}{\mathrm{part}}
\newcommand{\Rpac}{\mathrm{Rpac}}
\newcommand{\Relprime}{\mathrm{Relprime}}
\newcommand{\Ptn}{\mathrm{Part}}
\renewcommand{\phi}{\varphi}

\def\v{{\boldsymbol v}}

\def\t{{\boldsymbol t}}

\newcommand{\A}{\mathcal{V}}
\newcommand{\m}{\mathfrak{m}}

\newcommand\commentout[1]{}

\newtheorem{theorem}{Theorem}[section]
\newtheorem{problem}{Problem}[section]
\newtheorem{corollary}[theorem]{Corollary}
\newtheorem{proposition}[theorem]{Proposition}
\newtheorem{lemma}[theorem]{Lemma}

\theoremstyle{remark}
\newtheorem{example}[theorem]{Example}
\newtheorem{remark}[theorem]{Remark}
\theoremstyle{definition}
\newtheorem{definition}[theorem]{Definition}
\newtheorem{question}{Question}[section]

\newtheorem{setup}{Setup}[section]

\begin{document}

\title{Antichain Simplices}

\author{Benjamin Braun}
\address{Department of Mathematics\\
         University of Kentucky\\
         Lexington, KY 40506--0027}
\email{benjamin.braun@uky.edu}

\author{Brian Davis}
\address{Department of Mathematics\\
         University of Kentucky\\
         Lexington, KY 40506--0027}
\email{brian.davis@uky.edu}

\date{10 January 2019}



\maketitle

\begin{abstract}
  To each lattice simplex $\Delta$ we associate a poset encoding the additive structure of lattice points in the fundamental parallelepiped for $\Delta$.
  When this poset is an antichain, we say $\Delta$ is antichain.
  To each partition $\lambda$ of $n$, we associate a lattice simplex $\Delta_\lambda$ having one unimodular facet, and we investigate their associated posets.
  We give a number-theoretic characterization of the relations in these posets, as well as a simplified characterization in the case where each part of $\lambda$ is relatively prime to $n-1$.
  We use these characterizations to experimentally study $\Delta_\lambda$ for all partitions of $n$ with $n\leq 73$.
  Further, we experimentally study the prevalence of the antichain property among simplices with a restricted type of Hermite normal form, suggesting that the antichain property is common among simplices with this restriction.
  We also investigate the structure of these posets when $\lambda$ has only one or two distinct parts.
  Finally, we explain how this work relates to Poincar\'e series for the semigroup algebra associated to $\Delta$, and we prove that this series is rational when $\Delta$ is antichain.
\end{abstract}

\section{Introduction}\label{sec:intro}

Given a lattice simplex $\Delta$, the structure of the lattice points in the fundamental parallelepiped of the cone over $\Delta$ reflects a wealth of arithmetic and combinatorial properties of $\Delta$.
In this work, we study a partial order on these lattice points that encodes the additive relations among these points.
Thus, our main results are primarily arithmetic in nature, and will hopefully be of interest to those working in areas where fundamental parallelepipeds play a role, e.g. Ehrhart theory, partition identities, coding theory, optimization, etc.
Our motivation for this investigation comes from questions regarding rationality of Poincar\'e series for infinite graded resolutions of graded algebras, where it is of particular interest when the partial order associated to $\Delta$ has no relations.
Thus, after developing our main results using number-theoretic techniques, we explain their algebraic implications.

More precisely, in Section~\ref{sec:fpp} we define the fundamental parallelepiped poset $P(\Delta)$ associated to $\Delta$, where we say $\Delta$ is \emph{antichain} if $P(\Delta)\smallsetminus \{0\}$ has no relations.
To each partition $\lambda$ of $n$, we associate a lattice simplex $\Delta_\lambda$ having one unimodular facet, and we investigate the posets for these simplices in depth.
In Theorem~\ref{thm:order} we give a number-theoretic characterization of the relations in $P(\Delta_\lambda)$, and in Corollary~\ref{cor:order} we give a simplified characterization in the case where each part of $\lambda$ is relatively prime to $n-1$.
In Section~\ref{sec:experiments} we use these characterizations to generate empirical data, experimentally studying those $\Delta_\lambda$ for all partitions of $n$ with $n\leq 73$.
These experiments reveal that a substantial fraction of those $\lambda$ satisfying the relatively prime condition appear to have $\Delta_\lambda$ that are antichain.
Further, we experimentally study the prevalence of the antichain property among simplices with a restricted type of Hermite normal form, suggesting that the antichain property is common among simplices with this restriction.
In Section~\ref{sec:small} we shift perspective and investigate partitions having only one or two distinct parts.
Finally, in Section~\ref{sec:algebraic} we explain the algebraic implications of our work to Poincar\'e series of semigroup algebras associated to $\Delta$.
Specifically, we prove that if $\Delta$ is antichain, then the associated Poincar\'e series is rational.


\section{Fundamental Parallelepiped Posets and Their Structure}\label{sec:fpp}

\subsection{Lattice Simplices and Associated Posets}
Details regarding polytopes, cones, Hilbert bases, etc. as discussed in the following can be found in~\cite{CCD,MillerSturmfels}.
For a collection $\A=\{v_0,\dots,v_m\}$ of points in $\R^d$, we denote by $\conv(\A)$ the convex hull of $\A$.
In the case that $m=d$ and the set $\A^\circ\,:=\,\{(v_1-v_0),\dots,(v_d-v_0)\}$ is a vector space basis of $\R^d$, then we call $\Delta := \conv(\A)$ a $d$-simplex.
  We call the $v_i$'s the {\emph{vertices}} of $\Delta$, and if each $v_i$ is an integer point, i.e., lies in $\Z^d$, we call $\Delta$ a {\emph{lattice simplex}}.
We define the {\emph{conical hull}} of $\A$ to be the set
\[
  \cone(\A):=\left\{\sum_{i=0}^m\gamma_iv_i\,\text{ such that } 0\leq\gamma_i \right\}\subset \R^d. 
\]
Notice that the conical hull is unbounded, as in particular it contains the rays $\R_{\geq0}\cdot v_i$ for {${0\leq i \leq m}$}.

We are particularly interested in conical hulls of the following kind.
  Let $\A=\{v_0,\dots,v_d\}$ with $\Delta$ a lattice $d$-simplex.
  Then the {\emph{cone over}} $\Delta$ is the conical hull of the points $\{(1,v_0),\dots,(1,v_d)\}\subset\R^{d+1}$, and is denoted $\cone(\Delta)$.
We next recall the fundamental parallelepiped, a distinguished subset of $\cone(\Delta)$.

\begin{definition}For a lattice $d$-simplex $\Delta$ with vertices $v_0$ through $v_d$, the {\emph{fundamental parallelepiped}} $\Pi_\Delta$ is the set 
\[\Pi_\Delta\,:=\,\left\{\sum_{i=0}^{d}\gamma_i(1,v_i)\text{ such that }0\leq\gamma_i<1\right\}\subset\cone(\Delta).\]
\end{definition}

Interest in the fundamental parallelepiped $\Pi_\Delta$ arises mainly from the following well-known fact: every lattice point in $\cone(\Delta)$ can be written uniquely as a non-negative integer combination of the $(1,v_i)$'s and a lattice point in $\Pi_\Delta$.
To see this, note that because any element $z$ of $\cone(\Delta)\cap\Z^{d+1}$ lies in $\cone(\Delta)$, it is a non-negative linear combination of the $(1,v_i)$'s, i.e., there exist non-negative real coefficients $g_i$ such that
\[
  z=\sum_{i=0}^{d}g_i(1,v_i)=\left(\sum_{i=0}^{d}\left\lfloor g_i\right\rfloor (1,v_i)\right) + \left(\sum_{i=0}^{d}\{g_i\} (1,v_i)\right)
\]
where $\{g_i\}$ means the fractional part of $g_i$.
By setting $\gamma_i$ equal to $\{g_i\}$, we see that any point $z$ may be written  as a non-negative integral combination of the $(1,v_i)$'s and an integer point in $\Pi_\Delta\cap\Z^{d+1}$.
In particular, it is well-known that the set $\cone(\Delta)\cap\Z^{d+1}$ has a unique finite minimal additive generating set.

\begin{definition}\label{def:hb}
 The unique minimal additive generating set of $\cone(\Delta)\cap\Z^{d+1}$ is called the {\emph{Hilbert basis}} $\HH$ of $\cone(\Delta)$.
  It consists of the $(1,v_i)$'s and some lattice points  $h_1$ through $h_m$ in $\Pi_\Delta$ such that 
  \[
    \cone(\Delta)\cap\Z^{d+1}=\left\{\left(\sum_{i=0}^dr_i(1,v_i)\right)+\left(\sum_{j=1}^ms_i\,h_i\right)\text{ such that } r_i\,,\,s_j\in\Z_{\geq0}\right\} \, .
  \]
\end{definition}

The Hilbert basis consists of the cone generators $(1,v_i)$ together with the additively minimal elements $h_j$ of $\Pi_\Delta\cap\Z^{d+1}$.
 If the matrix whose columns are given by $(1,v_i)$ has determinant $\pm \v$, we say that the simplex $\Delta$ has {\emph{normalized volume}} $\v$. 
Since $\v$ is precisely the index of the sub-lattice generated by $(1,v_0)$ through $(1,v_d)$, we see that the normalized volume is equal to the number of lattice points in $\Pi_\Delta$.
If the normalized volume of $\Delta$ is equal to one, then we call $\Delta$ a {\emph{unimodular}} simplex. 

The set of lattice points $\Z^{d+1}\cap\Pi_\Delta$ can be equipped with the following partial order, inherited from a well-known partial order on the lattice points in $\Z^{d+1}\cap\cone(\Delta)$.
\begin{definition} The set $\Z^{d+1}\cap\Pi_\Delta$ is partially ordered by letting $\sigma\prec\mu$ if and only if $\mu-\sigma$ is an element of $\Z^{d+1}\cap\Pi_\Delta$. We call this poset the {\emph{fundamental parallelepiped poset}} $\PP(\Delta)$.
\end{definition}
Observe that the zero element of $\Z^{d+1}\cap\Pi_\Delta$ is below every other element of $\PP(\Delta)$, and that the minimal elements of $\PP(\Delta)\smallsetminus \{0\}$ are precisely the elements $h_1,\dots,h_m$ of the Hilbert basis of $\cone(\Delta)$.
Our interest is in the case where the Hilbert basis contains all the elements of $\PP(\Delta)\smallsetminus \{0\}$, leading to the following definition.

\begin{definition}
If $\Delta$ is a simplex such that $P(\Delta)\smallsetminus \{0\}$ is an antichain, we call $\Delta$ an \emph{antichain simplex} and say $\Delta$ is \emph{antichain}.
\end{definition}

\begin{example}
Recall that an empty simplex is one whose only lattice points are its vertices.
In the case that $\Delta$ is a $2$- or $3$-dimensional simplex, it is sufficient that $\Delta$ be empty in order for it to be antichain.
This follows since no lattice points of $\Pi_\Delta$ have $0$-th coordinate, i.e. height, equal to one, and thus the only possible $0$-th coordinates of lattice points in $\Pi_\Delta$ are $2$ or $3$.
However, sums of pairs of such lattice points have $0$-th coordinate equal to $4$, $5$, or $6$, and hence $\PP(\Delta)$ has no relations.
\end{example}

When attempting to determine whether or not $\Delta$ is antichain, the first problem encountered is to enumerate the elements of $\Pi_\Delta\cap\Z^{d+1}$.
As an initial attempt in this direction, let the matrix $A$ have columns given by $\{(1,v_i)\}_{0\leq i\leq d}$, where the $v_i$'s are the vertices of $\Delta$.
Recall that the normalized volume $\v$, the number of elements of $\Pi_\Delta$, may be computed by $\v=|\det A|$.
Recall also that the set $\Pi_\Delta$ is the image of $[0,1)^{d+1}$ under the linear transformation $A$, so that the preimage of a lattice point of $\Pi_\Delta$ must be a rational point of $[0,1)^{d+1}$ with denominator $\v$.
We may therefore compute the set of points in $\Pi_\Delta\cap\Z^{d+1}$ by considering each element of the form
\[
  \left\{A\cdot\left(\frac{b_0}{\v},\cdots,\frac{b_d}{\v}\right)^T\text{ such that }0\leq b_i<\v\right\} \, ,
\]
throwing out the ones which are not integer points.
Unfortunately, this test set grows as $\v^{d+1}$, and there is no easy way to describe the lattice points among them.

The software Normaliz~\cite{Normaliz} gives a more efficient implementation based on the fact that (possibly after a lattice translation) the matrix $A$ has a representation $A=UH$ where $U$ is a unimodular matrix and $H$ is in Hermite normal form. Bruns et al.~\cite{bruns2012power} show that, for $\{c_{i,i}\}_{0\leq i\leq d}$ given by the diagonal entries of the matrix $H$, lattice points in 
\[[0,c_{0,0})\times\cdots\times[0,c_{d,d})
\] are representatives of the quotient classes (in $\Z^{d+1}$ modulo the $(1,v_i)$'s) of the elements of $\Pi_\Delta\cap\Z^{d+1}$. It is then sufficient to consider the image under $A$ of the elements $\left(A^{-1}\cdot x\right)\mod\Z^{d+1}$ for $x\in[0,c_{0,0})\times\cdots\times[0,c_{d,d})$. This modular arithmetic is implemented in a computer easily enough, but introduces number theory to any analysis of the poset $P(\Delta)$.

\subsection{Lattice Simplices with a Unimodular Facet and their Posets}\label{sec:simplices}

In this work, we study a restricted class of simplices in order to avoid both of the methods described above for determining $\Pi_\Delta\cap \Z^{d+1}$.

\begin{definition} We say that a lattice $d$-simplex has a {\emph{unimodular facet}} if there exists a permutation $\pi$ in $\fS_{d+1}$ such that $\conv(\{v_{\pi_1},\dots,v_{\pi_d}\})$ is a unimodular lattice $(d-1)$-simplex.
\end{definition}

If $\Delta$ has a unimodular facet, then we may define a lattice preserving transformation taking $\Delta$ to $\conv(e_1,\dots,e_d,z)$ where the $e_i$ are the standard basis vectors of $\R^{d}$ and $z$ is a lattice point in $\Z^{d}$. Our goal in this chapter is to find a description of the relations in $P(\Delta)$ in terms of the coordinates of the point $z$.
To further simplify the situation, we consider only $z$ with positive entries.

\begin{definition}
  Let $\lambda=(\lambda_1,\dots,\lambda_d)$ be a lattice point in $\N^d$ such that $\sum_{i=1}^d\lambda_i=n$.
  We define $\Delta_\lambda:=\conv(e_1,\ldots,e_d,\lambda)\subset \R^d$ and use the shortened notation $\Pi_\lambda:=\Pi_{\Delta_\lambda}$ and $P(\lambda):=P(\Delta_\lambda)$.
\end{definition}

\begin{remark}
In the definition above, we can assume that $\lambda$ is a partition of $n$, as permuting the entries of $\lambda$ corresponds to a unimodular transformation of $\Delta_\lambda$.
  \end{remark}

\begin{remark}
  The simplices $\Delta_\lambda$ are defined in a similar manner to the simplices $\Delta_{(1,q)}$ that have recently been studied by multiple authors~\cite{BraunDavisPoinc,BraunDavisSolusIDP,BraunLiu,DavisMachine,Payne,SolusOsaka,SolusSimplices}.
  However, these are not the same families of simplices.
Specifically, the matrix giving the Hermite normal form of $\Delta_\lambda$ (after translating $\Delta_\lambda$ by $-e_1$) is
\[
  \left[
  \begin{array}{cccccc}
    0 & 1 & 0 & \cdots & 0 & \lambda_2 \\
        0 & 0 & 1 & \cdots & 0 & \lambda_3 \\
    \vdots & \vdots & \vdots & \ddots & \vdots & \vdots \\
    0 & 0 & 0 & \cdots & 1 & \lambda_d \\
    0 & 0 & 0 & \cdots & 0 & -1+\sum_i\lambda_i
    \end{array}\right] \, .
  \]
Note that $\sum_{i=2}^d\lambda_i\leq -1+\sum_{i=1}^d\lambda_i$.

Setting $Q=1+\sum_iq_i$, the Hermite normal form for $\Delta_{(1,q)}$ is
  \[
  \left[
  \begin{array}{cccccc}
    0 & 1 & 0 & \cdots & 0 & Q-q_2 \\
        0 & 0 & 1 & \cdots & 0 & Q-q_3 \\
    \vdots & \vdots & \vdots & \ddots & \vdots & \vdots \\
    0 & 0 & 0 & \cdots & 1 &Q-q_d \\
    0 & 0 & 0 & \cdots & 0 & Q
    \end{array}\right] \, .
\]
Note that for $d\geq 3$, we have $\sum_{i=2}^dQ-q_i>Q$.
Thus, these are distinct classes of simplices.
\end{remark}

\begin{remark}
Simplices with Hermite normal form having only a single non-trivial column, such as the ones given above, were previously considered by Hibi, Higashitani, and Li~\cite[Section 3]{HibiHigashitaniLi} in the context of Ehrhart theory.
\end{remark}

The following is a straightforward determinant calculation that also follows from the Hermite normal form given above.

\begin{proposition}\label{prop:fpplatticepoints}
The number of lattice points in $\Pi_\lambda$, which is equal to the normalized volume of $\Delta_\lambda$, is $\sum_{i=1}^d\lambda_i -1 = n-1$.
\end{proposition}

We can now describe the integer points in $\Pi_\lambda$ using only the entries of $\lambda$.

\begin{proposition}\label{prop:fppparameter}
For each integer $b$ with $0\leq b < n-1$, there is a unique lattice point $p(b)$ in $\Pi_\lambda$ given by 
\begin{equation}\label{eqn:p(b)}
p(b) =\left(\left(\sum_{i=1}^d\left\lceil \frac{b\lambda_i}{n-1}\right\rceil\right)-b\;,\; \left\lceil \frac{b\lambda_1}{n-1}\right\rceil\;,\;\dots\;,\;\left\lceil \frac{b\lambda_d}{n-1}\right\rceil\right) \, .
\end{equation}
Every integer point in $\Pi_\lambda$ arises in this manner, and thus we identify the integer $b$ with the lattice point $p(b)$.
\end{proposition}

\begin{proof}
For an element $\sum_{i=1}^d\gamma_i(1,e_i) + \gamma_{d+1}(1,\lambda) \in\Pi_\lambda\cap\Z^{d+1}$, we have
\[
\left(\left(\sum_{i=1}^{d+1}\gamma_i\right)\;,\;(\gamma_1+\gamma_{d+1}\lambda_1)\;,\;\dots\;,\;(\gamma_d+\gamma_{d+1}\lambda_d)\right) \in \Z^{d+1}\, .
\]
Because of the condition that each $\gamma_i$ is strictly less than one, for each $i$ we have
\[
\gamma_i=\lceil \gamma_{d+1}\lambda_i\rceil -  \gamma_{d+1}\lambda_i \, ,
\]
thus
\begin{align*}&\left(\gamma_{d+1}+\sum_{i=1}^{d}\left(\lceil \gamma_{d+1}\lambda_i\rceil-\gamma_{d+1}\lambda_i\right),\lceil \gamma_{d+1}\lambda_1\rceil,\dots,\lceil \gamma_{d+1}\lambda_d\rceil\right)\\
=&\left(\gamma_{d+1}\left(1-\sum_{i=1}^{d}\lambda_i\right)+\sum_{i=1}^d\lceil \gamma_{d+1}\lambda_i\rceil,\lceil \gamma_{d+1}\lambda_1\rceil,\dots,\lceil \gamma_{d+1}\lambda_d\rceil\right).\end{align*}

Observe that the first coordinate of this vector is an integer, hence
\[
\gamma_{d+1}\left(1-\sum_{i=1}^{d}\lambda_i\right)=\gamma_{d+1}(1-n) \in \Z \, .
\] 
It follows that $\gamma_{d+1}$ is a rational number of the form $b/(n-1)$, and every lattice point arises in this manner and is of the form
\[
\left(\left(\sum_{i=1}^d\left\lceil \frac{b\lambda_i}{n-1}\right\rceil\right)-b\;,\; \left\lceil \frac{b\lambda_1}{n-1}\right\rceil\;,\;\dots\;,\;\left\lceil \frac{b\lambda_d}{n-1}\right\rceil\right) \, .
\]
Since there are $n-1$ lattice points in $\Pi_\lambda$ by Proposition~\ref{prop:fpplatticepoints}, there must be one unique lattice point for each $0\leq b<n-1$.
\end{proof}

Using the notation from~\eqref{eqn:p(b)}, for $0\leq b<n-1$ we have that the zeroth coordinate of $p(b)$ is
\[
p(b)_0:=\left(\sum_{i=1}^d\left\lceil \frac{b\lambda_i}{n-1}\right\rceil\right)-b \, .
\]
Recall that we freely identify the integer $b$ with the lattice point $p(b)$.
The following lemma provides a connection between the parameterization of the integer points in $\Pi_\lambda$ and the order in $P(\lambda)$.

\begin{lemma}\label{lem:ij}
For $i,j\in P(\lambda)$ with $i\neq j$, we have $i\prec j$ if and only if $i<j$ and $p(i)+p(j-i)=p(j)$.
\end{lemma}

\begin{proof}
For the forward direction, if $i\prec j$, then by Proposition~\ref{prop:fppparameter} there exists a point $p(\ell)\in P(\lambda)$ such that $p(i)+p(\ell)=p(j)$.
Note that $\ell>0$ since $p(0)=0$.
It follows that for all $1\leq t\leq d$, we have
\[
\left\lceil \frac{i\lambda_t}{n-1}\right\rceil + \left\lceil \frac{\ell\lambda_t}{n-1}\right\rceil = \left\lceil \frac{j\lambda_t}{n-1}\right\rceil \, .
\]
Given this, we have that $p(i)_1+p(\ell)_1=p(j)_1$ reduces to $i+\ell=j$, forcing $\ell=j-i>0$, as desired.

For the reverse direction, if $i<j$ and $p(i)+p(j-i)=p(j)$, then we have $i\prec j$ by definition.
\end{proof}

We now give two propositions demonstrating how Lemma~\ref{lem:ij} can be used in practice.

\begin{proposition}\label{prop:twins}
If $i\prec j$ in $P(\lambda)$, then also $j-i\prec j$ in $P(\lambda)$.
\end{proposition}

\begin{proof}
By Lemma~\ref{lem:ij}, we have $i\prec j$ if and only if $i<j$ and $p(i)+p(j-i)=p(j)$ if and only if $j-i<j$ and $p(i)+p(j-i)=p(j)$ if and only if $j-i\prec j$.
\end{proof}

\begin{proposition}\label{prop:n-22}
Let $\lambda=(n-2,2)$.
Then $P(n-2,2)$ is equal to the following poset on the elements $\{1,2,\ldots,n-2\}$:
The minimal elements of $P(n-2,2)$ are $\{1,2,\ldots,\left\lfloor \frac{n-1}{2}\right\rfloor\}$ and the maximal elements are $\{\left\lfloor \frac{n-1}{2}\right\rfloor +1,\ldots, n-2\}$.
The cover relations are that the maximal element $\left\lfloor \frac{n-1}{2}\right\rfloor +j$ covers $\{j,j+1,\ldots,\left\lfloor \frac{n-1}{2}\right\rfloor\}$.
\end{proposition}

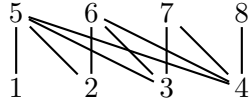
\begin{figure}[ht]
\begin{center}
\begin{tikzpicture}
    \node (10) at (0,0) {$5$};
    \node [right of = 10] (top6) {$6$};
    \node [right of = top6] (top7) {$7$};
    \node [right of = top7] (top8) {$8$};
    \node [below of = 10] (bot1) {$1$};
    \node [below of = top6] (bot2) {$2$};
    \node [below of = top7] (bot3) {$3$};
    \node [below of = top8] (bot4) {$4$};

    \draw [thick, shorten <=-2pt, shorten >=-2pt] (10) -- (bot1);
    \draw [thick, shorten <=-2pt, shorten >=-2pt] (10) -- (bot2);
    \draw [thick, shorten <=-2pt, shorten >=-2pt] (10) -- (bot3);
    \draw [thick, shorten <=-2pt, shorten >=-2pt] (10) -- (bot4);
    \draw [thick, shorten <=-2pt, shorten >=-2pt] (top6) -- (bot2);
    \draw [thick, shorten <=-2pt, shorten >=-2pt] (top6) -- (bot3);
    \draw [thick, shorten <=-2pt, shorten >=-2pt] (top6) -- (bot4);
    \draw [thick, shorten <=-2pt, shorten >=-2pt] (top7) -- (bot3);
    \draw [thick, shorten <=-2pt, shorten >=-2pt] (top7) -- (bot4);
    \draw [thick, shorten <=-2pt, shorten >=-2pt] (top8) -- (bot4);
\end{tikzpicture}
\caption{The poset $P(8,2)$.}
\end{center}
\end{figure}

\begin{proof}
By Lemma~\ref{lem:ij}, we see that $i\prec j$ if and only if $i<j$ and the following hold:
\begin{align}
\label{eq:2} \left\lceil \frac{2i}{n-1}\right\rceil + \left\lceil \frac{2(j-i)}{n-1}\right\rceil & = \left\lceil \frac{2j}{n-1}\right\rceil \\
\label{eq:n-2} \left\lceil \frac{i(n-2)}{n-1}\right\rceil + \left\lceil \frac{(j-i)(n-2)}{n-1}\right\rceil & = \left\lceil \frac{j(n-2)}{n-1}\right\rceil
\end{align}
It is straightforward to verify that these equations hold for the values claimed in the proposition statement.

To show that no other pairs $i<j$ lead to relations $i\prec j$, suppose that $1\leq i<j\leq \left\lfloor \frac{n-1}{2}\right\rfloor$.
Then in~\eqref{eq:2}, we obtain $1+1=1$, which is false.
Similarly, if $\left\lfloor \frac{n-1}{2}\right\rfloor +1 \leq i<j\leq n-2$, then in~\eqref{eq:2} we obtain $2+2=2$, which is again false.
\end{proof}


\subsection{Characterizing the Relations in $P(\lambda)$}

While Lemma~\ref{lem:ij} is a reasonable first tool, as Propositions~\ref{prop:twins} and~\ref{prop:n-22} illustrate, in general it is difficult to compute these relations directly.
Thus, we need to create a more sophisticated mechanism through which to study $P(\lambda)$.
In this section, we establish in Theorem~\ref{thm:order} a number-theoretic characterization of the relations in $P(\lambda)$.
Further, Corollary~\ref{cor:order} provides a particularly simple characterization in the case where each part of $\lambda$ is relatively prime to $n-1$.

For $0\leq i<n-1$, define the non-negative integers $r_{t,i}$ and $0\leq s_{t,i}<n-1$ by 
\begin{equation}
i\lambda_t=r_{t,i}(n-1)+s_{t,i} \, . \label{eqn:rsdef}
\end{equation}

\begin{lemma}\label{lem:scondition}
We have $i\prec j$ in $P(\lambda)$ if and only if $i<j$ and for every $t\in \{1,\ldots,d\}$ we have
\begin{equation}\label{eqn:seqn}
\frac{s_{t,i}+s_{t,j-i}-s_{t,j}}{n-1}=\left\lceil \frac{s_{t,i}}{n-1}\right\rceil + \left\lceil \frac{s_{t,j-i}}{n-1}\right\rceil - \left\lceil \frac{s_{t,j}}{n-1}\right\rceil 
\end{equation}
\end{lemma}

\begin{proof}
After adding and subtracting~\eqref{eqn:rsdef} for the values $i$, $j-i$, and $j$, we obtain
\begin{equation}\label{eqn:rtos}
r_{t,i}+r_{t,j-i}-r_{t,j} = \frac{-s_{t,i}-s_{t,j-i}+s_{t,j}}{n-1} \, .
\end{equation}
By dividing both sides of~\eqref{eqn:rsdef} by $n-1$ and taking the ceiling of both sides, we see that
\begin{equation}
\label{eqn:floor} \left\lceil \frac{\ell\lambda_t}{n-1}\right\rceil = r_{t,\ell}+\left\lceil \frac{s_{t,\ell}}{n-1}\right\rceil \, .
\end{equation}
Adding~\eqref{eqn:floor} with itself for $\ell$ equal to $i$ and $j-i$, then subtracting the equation with $\ell=j$, and further applying~\eqref{eqn:rtos},  we obtain
\begin{align*}
& \left\lceil \frac{i\lambda_t}{n-1}\right\rceil + \left\lceil \frac{(j-i)\lambda_t}{n-1}\right\rceil - \left\lceil \frac{j\lambda_t}{n-1}\right\rceil \\
 =& \, r_{t,i}+r_{t,j-i}-r_{t,j}+\left\lceil \frac{s_{t,i}}{n-1}\right\rceil + \left\lceil \frac{s_{t,j-i}}{n-1}\right\rceil - \left\lceil \frac{s_{t,j}}{n-1}\right\rceil \\
=& \, \frac{-s_{t,i}-s_{t,j-i}+s_{t,j}}{n-1}+\left\lceil \frac{s_{t,i}}{n-1}\right\rceil + \left\lceil  \frac{s_{t,j-i}}{n-1}\right\rceil - \left\lceil \frac{s_{t,j}}{n-1}\right\rceil \, .
\end{align*}
Recall that $i\prec j$ in $P(\lambda)$ if and only if $p(i)+p(j-i)=p(j)$ if and only if for all $t$, we have that 
\[
\left\lceil \frac{i\lambda_t}{n-1}\right\rceil + \left\lceil \frac{(j-i)\lambda_t}{n-1}\right\rceil - \left\lceil \frac{j\lambda_t}{n-1}\right\rceil =0 \, ,
\]
which by our computation above holds if and only if 
\[
\frac{s_{t,i}+s_{t,j-i}-s_{t,j}}{n-1}=\left\lceil \frac{s_{t,i}}{n-1}\right\rceil + \left\lceil  \frac{s_{t,j-i}}{n-1}\right\rceil - \left\lceil \frac{s_{t,j}}{n-1}\right\rceil \, .
\]
\end{proof}

\begin{theorem}\label{thm:order}
Let $\lambda$ be a partition of $n$.
We have $i\prec j$ in $P(\lambda)$ if and only if $i<j$ and for each $t\in \{1,\ldots,d\}$, one of the following holds:
\begin{enumerate}
\item $s_{t,i}> s_{t,j}>0$,
\item $s_{t,i}=0$ and $s_{t,j}=s_{t,j-i}$, or
\item $s_{t,j}=s_{t,i}>0$ and $s_{j-i}=0$.
\end{enumerate}
\end{theorem} 

\begin{proof}
\emph{Forward implication:}
Suppose that $i\prec j$ in $P(\lambda)$, and thus by Lemma~\ref{lem:scondition} the $s$-values satisfy~\eqref{eqn:seqn}.
We consider five cases:
\begin{itemize}
\item $s_{t,i}=0$
\item $s_{t,i}>s_{t,j}=0$
\item $s_{t,i}=s_{t,j}>0$
\item $s_{t,i}>s_{t,j}>0$
\item $s_{t,j}>s_{t,i}>0$
\end{itemize}

\emph{Case 1: $s_{t,i}=0$}.  
If $s_{t,i}=0$, then by~\eqref{eqn:seqn} we have that 
\[
\frac{s_{t,j-i}-s_{t,j}}{n-1} = \left\lceil \frac{s_{t,j-i}}{n-1}\right\rceil - \left\lceil \frac{s_{t,j}}{n-1}\right\rceil \, .
\]
Thus $\displaystyle \frac{s_{t,j-i}-s_{t,j}}{n-1}$ is equal to an integer, and the fact that $0\leq s_{t,\ell}<n-1$ implies that $s_{t,j-i}-s_{t,j}=0$.
Thus, we must have $s_{t,j-i}=s_{t,j}$.
This establishes the second condition in the theorem statement.

\emph{Case 2: $s_{t,i}>s_{t,j}=0$}. 
In this case,~\eqref{eqn:seqn} implies
\[
\frac{s_{t,i}+s_{t,j-i}}{n-1} = \left\lceil \frac{s_{t,i}}{n-1}\right\rceil + \left\lceil \frac{s_{t,j-i}}{n-1}\right\rceil \, .
\]
Thus $\displaystyle \frac{s_{t,i}+s_{t,j-i}}{n-1}$ is an integer, and again since $0\leq s_{t,\ell}<n-1$ and $0<s_{t,i}<n-1$ we have that
Thus, it is impossible to have $s_{t,i}>s_{t,j}=0$.

\emph{Case 3: $s_{t,i}=s_{t,j}>0$}. 
In this case,~\eqref{eqn:seqn} implies
\[
\frac{s_{t,j-i}}{n-1} = \left\lceil \frac{s_{t,j-i}}{n-1}\right\rceil \, .
\]
This forces $s_{t,j-i}=0$, resulting in the third condition in the theorem statement.

\emph{Case 4: $s_{t,i}>s_{t,j}>0$}.
In this case,~\eqref{eqn:seqn} implies $\displaystyle \frac{s_{t,i}+s_{t,j-i}-s_{t,j}}{n-1}$ is equal to an integer, and the fact that every $0\leq s_{t,\ell}<n-1$ implies this integer is $0$ or $1$.
Since $n-1>s_{t,i}-s_{t,j}>0$, we must have $\displaystyle \frac{s_{t,i}+s_{t,j-i}-s_{t,j}}{n-1}=1$, and also the right-hand side of~\eqref{eqn:seqn} is equal to $1$.
Thus, the first condition in the theorem statement is possible if $i\prec j$.

\emph{Case 5: $s_{t,j}>s_{t,i}>0$}.
Following the same logic as in the previous case, we must have $\displaystyle \frac{s_{t,i}+s_{t,j-i}-s_{t,j}}{n-1}=0$ and thus $s_{t,j-i}\neq 0$.
But then the right-hand side of~\eqref{eqn:seqn} is equal to $0$ while the right-hand side is equal to $1$, a contradiction.

\emph{Reverse implication:}
We verify that each of the three conditions listed in the theorem statement imply that~\eqref{eqn:seqn} is valid.

First, by equation~\eqref{eqn:rtos} we have $\displaystyle \frac{s_{t,i}+s_{t,j-i}-s_{t,j}}{n-1}\in \Z$.
Combining $n-1>s_{t,i}>s_{t,j}>0$ and the general bounds $0\leq s_{t,\ell}<n-1$ for all $\ell$, it follows that $\displaystyle \frac{s_{t,i}+s_{t,j-i}-s_{t,j}}{n-1}= 1$.
Thus, $s_{t,i}+s_{t,j-i}-s_{t,j}=n-1$.
Since $n-1>s_{t,i}-s_{t,j}>0$, we have $s_{t,j-i}=n-1-(s_{t,i}-s_{t,j})>0$, and thus
\[
\left\lceil \frac{s_{t,i}}{n-1}\right\rceil + \left\lceil \frac{s_{t,j-i}}{n-1} \right\rceil - \left\lceil \frac{s_{t,j-i}}{n-1}\right\rceil =1\, .
\] 
We conclude that equation~\eqref{eqn:seqn} holds.

Second, if $s_{t,i}=0$ and $s_{t,j}=s_{t,j-i}$, then it is immediate that~\eqref{eqn:seqn} holds.

Finally, if $s_{t,j}=s_{t,i}>0$ and $s_{j-i}=0$, then again it is immediate that~\eqref{eqn:seqn} holds.
\end{proof}

The following corollary illustrates a special case of Theorem~\ref{thm:order} that we will focus on in the remainder of this paper.
\begin{corollary}\label{cor:order}
Let $\lambda$ be a partition of $n$ where each coordinate is coprime to $n-1$, i.e. $\gcd(n-1,\lambda_t)=1$.
Then $i\prec j$ in $P(\lambda)$ if and only if $s_{t,i}> s_{t,j}>0$ for every $t$.
\end{corollary}

\begin{proof}
If $\gcd(n-1,\lambda_t)=1$, then $s_{t,i}\neq 0$ for all $i$. Thus, the second and third conditions in Theorem~\ref{thm:order} do not apply.
\end{proof}

\begin{remark}
Ehrhart-theoretic properties of simplices $\Delta_\lambda$ that satisfy the relatively prime condition in Corollary~\ref{cor:order} have been previously studied by Hibi, Higashitani, and Li~\cite[Section 3]{HibiHigashitaniLi}.
\end{remark}

We can use Corollary~\ref{cor:order} to prove the following structural result regarding $P(\lambda)$ in the case where each part of $\lambda$ is coprime to $n-1$.
\begin{theorem}\label{thm:duality}
Let $\lambda$ be a partition of $n$ such that each $\lambda_t$ is coprime to $n-1$.
Then $P(\lambda)$ is self-dual.
\end{theorem}

\begin{proof}
We claim that $\phi:x\to n-1-x$ for $x\in [n-2]$ is an order-reversing poset isomorphism.
It is clear that $\phi$ is a bijection.
To see that $\phi$ is order-reversing, observe that by Corollary~\ref{cor:order}, we have that $i\prec j$ if and only if 
\begin{equation}\label{eq:dual1}
s_{t,i}>s_{t,j} \text{ for all } t \, .
\end{equation}
Due to the fact that $\gcd(n-1,\lambda_t)=1$, we have that
$s_{t,i}+s_{t,n-i-i}=n-1$ for all $i$ and $t$, and thus~\eqref{eq:dual1} holds if and only if
\begin{equation}\label{eq:dual1}
s_{t,n-1-j}>s_{t,n-1-i} \text{ for all } t \, .
\end{equation}
This final condition holds if and only if $n-1-j\prec n-1-i$, as desired.
\end{proof}


\section{Experimental Results}\label{sec:experiments}

\subsection{Exhaustive search of $\Delta_\lambda$ over all partitions of $n$}

The results in Section~\ref{sec:fpp}, particularly Theorem~\ref{thm:order} and Corollary~\ref{cor:order}, provide explicit tools for studying the relations in $P(\lambda)$.
Also, Corollary~\ref{cor:order} and Theorem~\ref{thm:duality} demonstrate that the condition that the parts of $\lambda$ be relatively prime to $\sum_i\lambda_i-1$ imposes additional structure on $P(\lambda)$, leading us to the following definitions.

Given a partition $\lambda$ of $n$, if $\gcd(\lambda_i,n-1)=1$ for all $i$, we say $\lambda$ \emph{satisfies the relatively prime condition}.
  Let $\Ptn(n)$ denote the set of partitions of $n$, and let $\ptn(n):=|\Ptn(n)|$.
Let $\Relprime(n)$ denote the set of partitions of $n$ that satisfy the relatively prime condition, and set $\relprime(n):=|\Relprime(n)|$.
Finally, let $\Rpac(n)$ denote the subset of $\Relprime(n)$ for which $\Delta_\lambda$ is an antichain simplex, and set $\rpac(n):=|\Rpac(n)|$.

Using SageMath~\cite{SAGE} via CoCalc.com~\cite{cocalc}, we computed $\ptn(n)$, $\relprime(n)$, and $\rpac(n)$ for all $1\leq n\leq 73$; the results are given in the table in Appendix~\ref{sec:appendix}.
Figure~\ref{fig:percentrelprime} plots the ratio $\rp(n)/\pn(n)$ for these values, and Figure~\ref{fig:percentantichainofrelprime} plots the ratio $\rpac(n)/\relprime(n)$.

\begin{figure}[h]
\begin{minipage}{0.5\textwidth}
  \centering
\resizebox{\linewidth}{!}{\input{percent_relatively_prime.pgf}}
\caption{The ratio $\rp(n)/\pn(n)$ for $1\leq n\leq 73$.}
  \label{fig:percentrelprime}
\end{minipage}%
\begin{minipage}{0.5\textwidth}
  \centering
\resizebox{\linewidth}{!}{\input{percent_antichain_of_relativelyprime.pgf}}
\caption{The ratio $\rpac(n)/\relprime(n)$ for $1\leq n\leq 73$.}
  \label{fig:percentantichainofrelprime}
\end{minipage}
\end{figure}

What follows are some observations regarding this experimental data.

\begin{enumerate}
\item Figure~\ref{fig:percentantichainofrelprime} shows that regardless of the total value of $\relprime(n)$, the ratio $\rpac(n)/\relprime(n)$ appears to be generally above $0.8$ and as $n$ grows it is clustering between $0.85$ and $0.95$.
  Thus, these experiments suggest that when $\lambda$ satisfies the relatively prime condition, it is likely that $\Delta_\lambda$ is antichain.
\item Figure~\ref{fig:percentrelprime} shows that when $n-1$ is not prime or the square of a prime, the ratio $\rp(n)/\pn(n)$ appears to be small, and thus our consideration of the relatively prime condition does not broadly apply to partitions in this case.
  However, it is immediate that when $n-1$ is prime, every partition of $n$ except for $1+(n-1)$ satisfies the relatively prime condition, and thus $\rpac(n)/\relprime(n)=\rpac(n)/(\pn(n)-1)$.
  Thus, it appears that one likely source of antichain simplices are those $\Delta_\lambda$ for which $n-1$ is prime.
\item In Figure~\ref{fig:percentantichainofrelprime}, the values of $\rpac(n)/\relprime(n)$ for $n\geq 13$ that lie on the upper hull of the data plot arise from $n$ in $\{13, 19, 31, 43, 61, 67, 73\}$.
  These are all prime numbers, and an OEIS~\cite{oeis} search finds that these values arise in three known sequences, including the sequence A040047 of those primes $p$ such that $x^3=6$ has no solution mod $p$.
\item Again in Figure~\ref{fig:percentantichainofrelprime}, the values of $\rpac(n)/\relprime(n)$ for $70\geq n\geq 20$ that lie on the lower hull of the data plot arise from $n$ in $\{20, 26, 32, 38, 44, 50, 62, 68\}$.
  These values are all of the form $6k+2$, though whether by coincidence or mathematics it is not clear.
\item When $n-1$ is a superabundant (OEIS A004394) or highly composite (OEIS A002182) number, one might expect to see particularly low numbers of relatively prime antichain simplices, which is supported by the data given in Appendix~\ref{sec:appendix}.
\end{enumerate}

At this time, the authors do not have an explanation for why the ratio $\rpac(n)/\relprime(n)$ appears to be clustering as $n$ grows, leading to the following problems.

\begin{problem}
Determine if there is a limiting value to which the sequence $\rpac(n)/\relprime(n)$ converges as $n$ increases.
Alternatively, determine if there are any connections between a liminf or limsup value for $\rpac(n)/\relprime(n)$ and the values of $n$ corresponding to subsequences achieving those values, as hinted at in the observations above.
\end{problem}

\subsection{Random sampling of simplices with one non-trivial column in Hermite normal form}

It is worthwhile to compare the results for our restricted $\Delta_\lambda$ simplices to arbitrary simplices with Hermite normal form given by
\[
  \left[
  \begin{array}{cccccc}
    0 & 1 & 0 & \cdots & 0 & a_1 \\
        0 & 0 & 1 & \cdots & 0 & a_2 \\
    \vdots & \vdots & \vdots & \ddots & \vdots & \vdots \\
    0 & 0 & 0 & \cdots & 1 & a_{d-1} \\
    0 & 0 & 0 & \cdots & 0 & n
    \end{array}\right] 
  \]
where for $i=1,\ldots,d-1$ we have $0\leq a_i<n$.
We will call a simplex of this form a \emph{one-column (n,d) Hermite normal form simplex}.
Let $OCH(n,d)$ denote the family of one-column $(n,d)$ Hermite normal form simplices, and let 
\[
OCH^+(n,d):=\{ A\in OCH(n,d): 1\leq a_i<n \text{ for all } i\} \, .
\]
Thus, $OCH^+(n,d)$ contains those simplices in $OCH(n,d)$ that are not obviously arising as lattice pyramids over simplices of smaller dimension.

There are $(n-1)^{d-1}$ simplices in $OCH^+(n,d)$.
For $67$ random choices (uniform without replacement) of $(n,d)\in \{3,\ldots,20\}\times \{3,\ldots,20\}$, we selected $n^{3}$ random samples from $OCH^+(n,d)$ and computed the resulting fraction $f(n,d)$ of antichain simplices in this sample.
We plotted the points $(n/d,f(n,d))$ in Figure~\ref{fig:fracantichain}.
It is particularly interesting that when $d$ is large relative to $n$, the percentage of antichain simplices among those sampled appears to be close to $1$, leading to the following problem.

\begin{problem}
Fix $n\geq 2$.
Is it true that the fraction of antichain simplices in $OCH^+(n,d)$ goes to $1$ as $d\to \infty$?
Alternatively, let $ac_n(d)$ denote the fraction of $\displaystyle \bigcup_{j=3}^dOCH^+(n,j)$ that are antichain simplices; what is the liminf of $ac_n(d)$ as $d\to \infty$?
\end{problem}

\begin{figure}[h]
\centering
  \centering
  \includegraphics[width=0.65\linewidth]{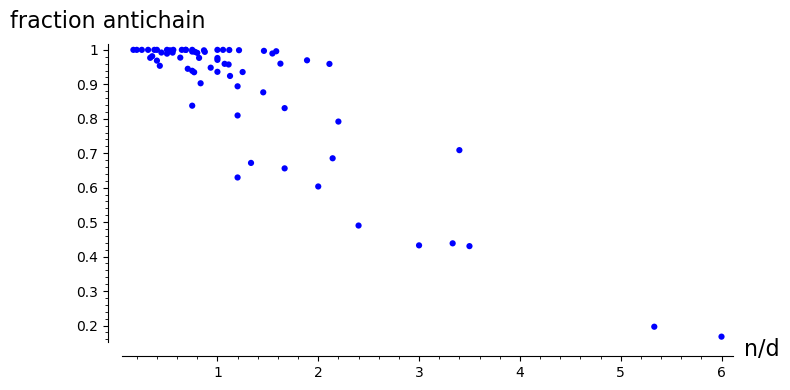}
  \caption{For $67$ random choices (uniform without replacement) of $(n,d)\in \{3,\ldots,20\}\times \{3,\ldots,20\}$, we plot the point $(n/d,f(n,d))$ where $f(n,d)$ is the fraction of antichain simplices among $n^{3}$ random samples from $OCH^+(n,d)$.}
  \label{fig:fracantichain}
\end{figure}


\section{Partitions With a Small Number of Parts}\label{sec:small}

Rather than consider all partitions of $n$ as we did for our experimental data, in this section we consider $\Delta_\lambda$ for $\lambda$ having only one or two distinct parts.
The results in this section illustrate the complications involved in computing $P(\lambda)$ even in relatively ``simple'' cases for $\lambda$ that satisfy the relatively prime condition.

\subsection{Partitions With One Distinct Part}\label{sec:onedistinct}

When $\lambda=(x,x,\ldots,x)$ has $v$ occurrences of $x$, it is immediate that $x$ is coprime to $n-1=vx-1$.
In this case, $P(\lambda)$ has a direct interpretation as a subposet of $\Z^2$.

\begin{theorem}\label{thm:onepartposet}
For $\lambda=(x,x,\ldots,x)$ with $v$ occurrences of $x$, we have that $P(\lambda)$ is isomorphic to the poset with elements 
\[
\left\{(r,p):0\leq r<x, 0\leq p<v \right\}\setminus\left\{(0,0),(x-1,v-1) \right\} 
\]
and order relation $(r,p)\prec (r',p')$ if both $p>p'$ and $r'>r$.
\end{theorem}

\begin{proof}
For $1\leq i\leq vx-2$, write 
\[
i=r_i v+p_i
\]
where $0\leq r_i<x$ and $0\leq p_i<v$, but we do not have simultaneously $r_i=x-1$ and $p_i=v-1$.
Then
\begin{align*}
s_i & = ix-\left\lfloor \frac{ix}{xv-1}\right\rfloor(xv-1) \\
& = x(r_i v+p_i) - \left\lfloor \frac{(r_i v + p_i)x}{xv-1}\right\rfloor(xv-1) \\
& = xr_i v + xp_i - \left\lfloor \frac{xr_i v  -r_i + r_i + p_i x}{xv-1}\right\rfloor(xv-1) \\ 
& = xr_i v + xp_i - \left( r_i +\left\lfloor \frac{r_i + p_i x}{xv-1}\right\rfloor\right)(xv-1) \\ 
& = r_i + xp_i - \left\lfloor \frac{r_i + p_i x}{xv-1}\right\rfloor(xv-1) \\ 
& = r_i + xp_i
\end{align*}
where the final equality is a result of the bounds on $r_i$ and $p_i$ forcing the floor function to be zero.
Thus, if $i=r_i v+p_i$ and $j=r_j v +p_j$, then we have $i\prec j$ in $P(\lambda)$ if and only if $i<j$ and $s_i>s_j$, which happens if and only if the following two conditions simultaneously occur:
\begin{itemize}
\item $p_i > p_j$ or $p_i=p_j$ with $r_i>r_j$
\item $r_j >r_i$ or $r_j = r_i$ with $p_j > p_i$
\end{itemize}
The only way for both conditions to simultaneously occur is to have $p_i > p_j$ and $r_j >r_i$, and thus our proof is complete.
\end{proof}

The following corollary follows immediately.

\begin{corollary}\label{cor:onedigitswap}
The posets for $\lambda=(x,x,\ldots,x)$ where $x$ occurs $v$ times and $\lambda'=(v,v,\ldots,v)$ where $v$ occurs $x$ times are isomorphic.
\end{corollary}

Corollary~\ref{cor:onedigitswap} is interesting because the two lattice simplices corresponding to $\lambda$ and $\lambda'$ are in different dimensions.
As an aside, we remark that the order on the lattice points within a rectangular grid given in Theorem~\ref{thm:onepartposet} corresponds to the reflexive closure of the direct product of two strict total orders.

\begin{example}
Figures~\ref{fig:46} and~\ref{fig:64} show the Hasse diagrams of the posets $P(4,4,4,4,4,4)$ and $P(6,6,6,6)$, respectively, embedded in $\Z^2$ as described in Theorem~\ref{thm:onepartposet}.
This illustrates the isomorphism obtained by switching the roles of $x$ and $v$.
\begin{figure}[ht]
\centering
\begin{minipage}{.5\textwidth}
  \centering
\resizebox{0.5\linewidth}{!}{\input{posetfig46.pgf}}
\caption{$P(4,4,4,4,4,4)$.}
  \label{fig:46}
\end{minipage}%
\begin{minipage}{.5\textwidth}
  \centering
\resizebox{0.8\linewidth}{!}{\input{posetfig64.pgf}}
\caption{$P(6,6,6,6)$.}
  \label{fig:64}
\end{minipage}
\end{figure}
\end{example}

\subsection{Partitions With Two Distinct Parts}\label{sec:twodistinct}

The situation for $\lambda$ with two distinct parts is significantly more complicated than for one distinct part.
Rather than consider arbitrary pairs of distinct parts for $\lambda$, we will consider the special case where one of the parts is a multiple of the other.
Our main result of this subsection is Theorem~\ref{thm:conematrix}, and we give an example illustrating how it can be applied to construct our posets.
Specifically, we use the following setup.

\begin{setup}\label{setup:xax}
Let $\lambda=(x,\dots,x,ax,\dots,ax)$ with $3\leq a\leq x$ where the multiplicity of $x$ is $ua+v$ and the multiplicity of $ax$ is $v-(u+1)$.
This places an implicit restriction on the values of $u$ and $v$ as follows:
\begin{align*}
0&\leq u\leq a-3\\
u+2&\leq v\leq \min\left(a-1,\frac{a(x-1)}{x}\right) \, .
\end{align*} 

Let $n=|\lambda|$, so that 
\begin{equation}\label{eqn:n-1}
n-1=x[(a+1)(v-1)+1] -1=(xa(v-1))+xv-1 \, .
\end{equation}
\end{setup}

For $0\leq i\leq n-2$, define as usual
\[
s_{1,i}:=ix-(n-1)\left\lfloor\frac{ix}{n-1}\right\rfloor \text{ and } \, \,
s_{2,i}:=iax-(n-1)\left\lfloor\frac{iax}{n-1}\right\rfloor \, .
\]

As in the proof of Theorem~\ref{thm:onepartposet}, our analysis will require us to represent $i$ as a quotient with remainder.
In this case, we will use a combination of two quotients-with-remainder from applying the division algorithm twice.
Observing that $n/x=(a+1)(v-1)+1$, we write $0\leq i\leq n-2$ as 
\begin{equation}\label{eqn:newi}
i=\frac{n}{x}r_i+(v-1)p_i+q_i
\end{equation}
subject to the following inequalities:
\begin{equation}\label{eqn:rbound}
0\leq r_i<x
\end{equation}
with
\[
0 \leq (v-1)p_i+q_i<\displaystyle n/x =(a+1)(v-1)+1\, ,
\]
and
\begin{equation}\label{eqn:pbound}
0\leq p_i< a+2
\end{equation}
with
\begin{equation}\label{eqn:qbound} 
0 \leq q_i< v-1  \, ,
\end{equation}
where $p_i=a+1$ implies $q_i=0$.

Our first goal is to express $s_{1,i}$ and $s_{2,i}$ as explicit functions of $r_i,p_i,q_i,x,a$, and $v$.

\begin{lemma} 
$s_{1,i}=r_i+x\big[(v-1)p_i+q_i\big]$.
\end{lemma}
\begin{proof}
By~\eqref{eqn:newi}, since $ix=nr_i+x\big[(v-1)p_i+q_i\big]$, we have that 
\begin{align*} 
s_{1,i}&=nr_i+x\big[(v-1)p_i+q_i\big]-(n-1)\left\lfloor\frac{nr_i+x\big[(v-1)p_i+q_i\big]}{n-1}\right\rfloor\\
&=r_i+x\big[(v-1)p_i+q_i\big]-(n-1)\left(-r_i+\left\lfloor\frac{nr_i+x\big[(v-1)p_i+q_i\big]}{n-1}\right\rfloor\right)\\
&=r_i+x\big[(v-1)p_i+q_i\big]-(n-1)\left\lfloor\frac{r_i+x\big[(v-1)p_i+q_i\big]}{n-1}\right\rfloor \, .
\end{align*}
Observe that equations~\eqref{eqn:n-1}, \eqref{eqn:qbound}, \eqref{eqn:pbound}, and~\eqref{eqn:rbound} imply that
\[
0\leq r_i+x\big[(v-1)p_i+q_i\big]\leq x-1 + x(a+1)(v-1)=n-1 \, ,
\] 
with equality only if $r_i=x-1$ and $p_i=a+1$ simultaneously. 
But in this case, we have that 
\[
ix=(x-1)n+x(a+1)(v-1)=(n-1)x \, ,
\] 
a contradiction with $i\leq n-2$. 
Thus the right hand floor term is zero in our expression for $s_{1,i}$, and the result follows.\end{proof}

Define the function 
\[
f(i):=\,ar_i-(xv-1)p_i+xaq_i \, ,
\] 
and associated set partition $[n-2]=\uplus_k F_k$ given by 
\[
F_k:=\left\{i: k= -\left\lfloor \frac{f(i)}{n-1}\right\rfloor \right\} \, .
\]

\begin{lemma} 
If $i$ is in $F_k$, then $ s_{2,i}=f(i)+k(n-1)$.
\end{lemma}
\begin{proof} 
Observe that 
\[
\frac{s_{2,i}}{n-1}=\left\{\frac{iax}{n-1}\right\} \, .
\] 
Further, notice that using~\eqref{eqn:n-1} we have
\begin{align*}
axi-f(i)&=axi-\left(ar_i-(xv-1)p_i+xaq_i\right) \\
&=a(n-1)r_i+p_i(ax(v-1)+(xv-1))\\
&=(n-1)(ar_i+p_i) \, ,
\end{align*}
 an integer multiple of $n-1$, so that 
\[
\left\{\frac{iax}{n-1}\right\}=\left\{\frac{f(i)}{n-1}\right\} \, .
\]

It follows that 
\[
s_{2,i}=(n-1)\left(\frac{f(i)}{n-1}-\left\lfloor\frac{f(i)}{n-1}\right\rfloor\right)=f(i)+k(n-1) \, .
\]
since $i\in F_k$.
\end{proof}

\begin{theorem}\label{thm:conematrix}
Suppose that $n-1$ is coprime to both $x$ and $a$, and let $\ell\in \Z$.
If $i\in F_k$ and $j\in F_{k+\ell}$, then $i\prec j$ if and only if $(p_j-p_i,q_j-q_i,r_j-r_i)$  lies in the open polyhedral cone defined by $Cx>(\ell(n-1),0,0)^T$, where $C$ is the matrix
\[C:=\begin{bmatrix} 
xv-1 & -ax&-a \\
1-v&-1&0\\
0&0&1
\end{bmatrix}.\]
 \end{theorem}
 \begin{proof}
 Summarizing our results from the above lemmas,
\begin{enumerate}
\item $i<j$ if and only if $r_i<r_j$ or $r_i=r_j$ and $(v-1)p_i+q_i<(v-1)p_j+q_j$.\smallskip

\item $s_{1,i}>s_{1,j}$ if and only if $r_i-r_j>x\big[(v-1)(p_j-p_i)+(q_j-q_i)\big]$. If $i<j$, then $r_i-r_j<0$ or $r_i-r_j=0$ and $(v-1)(p_j-p_i)+q_j-q_i>0$. Since $-(x-1)\leq r_i-r_j\leq x-1$, we see that for $i<j$, we have $s_{1,i}>s_{1,j}$ if and only if $r_i<r_j$ and $\big[(v-1)p_i+q_i\big] >\big[(v-1)p_j+q_j\big]$. \smallskip
\item If $i$ is in $F_k$ and $j$ is in $F_{k+\ell}$, then $s_{2,i}>s_{2,j}$ if and only if $f(i)>f(j)+\ell(n-1)$, i.e., if and only if $
      (xv-1)(p_j-p_i)-ax(q_j-q_i)+a(r_j-r_i)> \ell(n-1)$.
\end{enumerate}

Notice that these conditions correspond to affine half-spaces and are simultaneously satisfied exactly when $(p_j-p_i,q_j-q_i,r_j-r_i)\in \Z^3$  lies in the open polyhedral cone defined by the matrix equation $Cx>(\ell(n-1),0,0)^T$.
\end{proof}

The following proposition shows that there are a limited number of values of $k$ for which $F_k$ is non-empty.

\begin{proposition}
$\displaystyle [n-2]=F_0\uplus F_1\uplus F_2$
\end{proposition}

\begin{proof}
We show that $-2(n-1)<f(i)<n-1$ for every $i\in [n-2]$, from which the result follows.
To prove $-2(n-1)<f(i)$, we observe the following, using equations~\eqref{eqn:rbound}, \eqref{eqn:qbound}, and~\eqref{eqn:pbound} for the first inequality and the fact that (by definition) $v\geq 2$ and $a\geq 1$ for the second inequality:
\begin{align*}
f(i)+2(n-1) &= ar_i-p_i(xv-1)+xaq_i+2(xa(v-1)+xv-1) \\
& = a(r_i+xq_i+2x(v-1))-(xv-1)(p_i-2) \\
 &\geq  a(2x(v-1))-(xv-1)(a-1) \\
& = (v-2)ax+xv+a-1 \\
& \geq 0
\end{align*}

To prove the $f(i)<n-1$, we observe that using the same inequalities as before together with~\eqref{eqn:n-1} we have:
\begin{align*}
n-1-f(i) & = n-1 - ar_i+(xv-1)p_i-xaq_i \\
& \geq n-1 -a(x-1)-xa(v-2) \\
& = xa(v-1)+xv-1-a(x-1)-xa(v-2) \\
& = a-1+xv \\
& \geq 0
\end{align*}
\end{proof}

To illustrate Theorem~\ref{thm:conematrix} in a specific application, let $v=2$, so that $\lambda=(x,x,ax)\in\Z^3$.
Thus, $n=x(a+2)$ and $i=(a+2)r_i+p_i$, with $0\leq p_i<a+2$, and $q_i=0$ for all $i$.
In this case, Theorem \ref{thm:conematrix} is equivalent to $i\prec j$ if and only if $i\in F_s$, $j\in F_{s+\ell}$, and 
\[\begin{bmatrix}
   2x-1 & -ax & -a \\
   -1 & -1 & 0\\
   0&0&1
   \end{bmatrix}
   \begin{bmatrix}
   p_j-p_i\\
   0\\
   r_j-r_i
   \end{bmatrix}
   =
   \begin{bmatrix}
   (2x-1)(p_j-p_i)-a(r_j-r_i) \\
   p_i-p_j\\
   r_j-r_i
   \end{bmatrix}
   >
   \begin{bmatrix}
   \ell(n-1)\\
   0\\
   0
   \end{bmatrix}
.\]
The following two propositions are used in Example~\ref{fppExample} that follows.

\begin{proposition}For $v=2$, \[
r_i+r_{n-1-i}=x-1\quad\text{ and }\quad p_i+p_{n-1-i}=a+1.\]
\end{proposition}
\begin{proof}Notice that \begin{align*}
n&=i+(n-i-1) +1\\
&=(a+2)(r_i+r_{n-1-i})+p_i+p_{n-1-i}+1\\
&=(a+2)x,
\end{align*}
so that \[x=(r_i+r_{n-1-i})+\frac{p_i+p_{n-1-i}+1}{a+2}.\]
Since $x$ is an integer, this implies that $(p_i+p_{n-1-i})\mod(a+2)\equiv a+1$ and $p_i+p_{n-1-i} = a+1 + k(a+2)$. Since $0\leq p_j<a+2$, the result follows.
\end{proof}
\begin{proposition}For $v=2$, 
\[i\in F_s\quad\text{ if and only if }\quad n-1-i\in F_{2-s}.
\]
\end{proposition}
\begin{proof}Since by definition $i\in F_s$ if and only if \[
s=-\left\lfloor\frac{ar_i-(2x-1)p_i)}{n-1}\right\rfloor
,\] the proposition is equivalent to the claim that \[
-\left\lfloor\frac{ar_i-(2x-1)p_i}{n-1}\right\rfloor -\left\lfloor\frac{ar_{n-1-i}-(2x-1)p_{n-1-i}}{n-1}\right\rfloor=2
.\]
Using the previous proposition and some tedious but straightforward algebra, we have that
\begin{align*}
&-\left\lfloor\frac{ar_i-(2x-1)p_i}{n-1}\right\rfloor -\left\lfloor\frac{ar_{n-1-i}-(2x-1)p_{n-1-i}}{n-1}\right\rfloor\\
&=-\left\lfloor\frac{ar_i-(2x-1)p_i}{n-1}\right\rfloor -\left\lfloor\frac{a(x-1-r_i)-(2x-1)[(a+1)-p_i]}{n-1}\right\rfloor\\
&=-\left\lfloor\frac{ar_i-(2x-1)p_i}{n-1}\right\rfloor -\left\lfloor\frac{-(n-1)-[ar_i-(2x-1p_i]}{n-1}\right\rfloor\\
&=-\left\lfloor\frac{ar_i-(2x-1)p_i}{n-1}\right\rfloor +\left\lceil\frac{(n-1)+[ar_i-(2x-1)p_i]}{n-1}\right\rfloor\\
&=1+\left(\left\lceil\frac{ar_i-(2x-1)p_i}{n-1}\right\rfloor-\left\lfloor\frac{ar_i-(2x-1)p_i}{n-1}\right\rfloor\right),
\end{align*}
so that unless $ar_i-(2x-1)p_i$ is a multiple of $(n-1)$, the claim holds.

Let $ar_i-(2x-1)p_i$ equal $k(n-1)$; we will show that $k$ is not an integer.
Observe that $p_i=i-(a+2)r_i$ (by definition) and that $n=x(a+2)$. We obtain
\begin{align*}k(n-1)&=ar_i-(2x-1)p_i\\
&=ar_i-(2x-1)(i-(a+2)r_i)\\
&=r_i[a+(2x-1)(a+2)]-(2x-1)i\\
&=r_i[2x(a+2)-2]-(2x-1)i\\
&=2r_i(n-1)-(2x-1)i,
\end{align*} 
so that 
\begin{align*}
k&=2r_i-\frac{(2x-1)i}{n-1}\\
&=2r_i-i\frac{(n-1)-ax}{n-1}\\
&=2r_i-i+\frac{iax}{n-1}.
\end{align*}
Since we assume that both $a$ and $x$ are relatively prime to $n-1$ and $i$ is less than $n-1$, $k$ is not an integer.
\end{proof}

We next give an example to demonstrate how to use our results to construct $P(x,x,ax)$.

\begin{example}\label{fppExample}Let $a=x=3$, so that $n-1$ is equal to 14, and note that this is relatively prime to 3. Since $i$ is in $F_0$ if and only if $ar_i\geq(2x-1)p_i$, we draw the elements of $P(3,3,3\cdot3)\smallsetminus \{0\}$ in the plane as shown in Figure~\ref{fig:construct339}, where the diamonds correspond to elements of $F_0$ and the triangles correspond to elements of $F_2$.
  
\begin{figure}[ht]
\centering
\begin{minipage}{.5\textwidth}
  \centering
\begin{tikzpicture}

    \node[diamond,fill=black,minimum size=10pt,inner sep=0pt] (10) at (0,0) {};
    \node[diamond,fill=black,minimum size=10pt,inner sep=0pt] [right of = 10] (11) {};
    \node[circle,fill=black,minimum size=10pt,inner sep=0pt] [right of = 11] (12) {};
    \node[circle,fill=black,minimum size=10pt,inner sep=0pt] [right of = 12] (13) {};
    
    \node[diamond,fill=black,minimum size=10pt,inner sep=0pt] [below of = 10] (5) {};
    \node[circle,fill=black,minimum size=10pt,inner sep=0pt] [below of = 11] (6) {};
    \node[circle,fill=black,minimum size=10pt,inner sep=0pt] [below of = 12] (7) {};
    \node[circle,fill=black,minimum size=10pt,inner sep=0pt] [below of = 13] (8) {};
    \node[regular polygon,regular polygon sides=3,fill=black,minimum size=10pt,inner sep=0pt] [right of = 8] (9) {};
    
    \node[regular polygon,regular polygon sides=3,fill=black,minimum size=10pt,inner sep=0pt] [below of = 9] (4) {};
    \node[regular polygon,regular polygon sides=3,fill=black,minimum size=10pt,inner sep=0pt] [below of = 8] (3) {};
    \node[circle,fill=black,minimum size=10pt,inner sep=0pt] [below of = 7] (2) {};
    \node[circle,fill=black,minimum size=10pt,inner sep=0pt] [below of = 6] (1) {};
	\node[circle,minimum size=10pt,inner sep=2pt,draw] [below of = 5] (0) {};
	\node [right of = 4] (right) {$p_i$};
	\node [above of = 10] (up) {$r_i$};
	
	\draw[->] (0) -- (right);
	\draw[->] (0) -- (up);
	\draw[->, dashed] (0) -- (3/2,0);
	\draw[->, dashed] (5/2,-2) -- (4,0);
	
	\node at (-.5,0) {2};
	\node at (-.5,-1) {1};
	\node at (-.5,-2) {0};
	
	\node at (1,-2.25) [below]{1};
	\node at (0,-2.25) [below]{0};
	\node at (2,-2.25) [below]{2};
	\node at (3,-2.25) [below]{3};
	\node at (4,-2.25) [below]{4};


\end{tikzpicture}
\caption{Constructing the poset $P(3,3,9)$.}
\label{fig:construct339}
\end{minipage}%
\begin{minipage}{.5\textwidth}
  \centering
\begin{tikzpicture}

    \node[circle,fill=black,minimum size=10pt,inner sep=0pt] at (-1,1) {};
    \node[circle,fill=black,minimum size=10pt,inner sep=0pt] at (-2,1) {};
    \node[circle,fill=black,minimum size=10pt,inner sep=0pt] at (-1,2) {};
    \node at (-1.5,0) [below]{$p_j-p_i=0$};
    \node at (0,5/3) [right]{$r_j-r_i=0$};
    \node at (-4,7/3) {$3(r_j-r_i)=14+5(p_j-p_i)$};

	\draw[->,dashed] (0,0) -- (-3.5,0);
	\draw[->,dashed] (0,0) -- (0,5);
	\draw[<->,dashed] (-3,-1/5) -- (1/3,5);
\end{tikzpicture}
\caption{Relations in $P(3,3,9)$ for $\ell=-1$}
\label{relsP1}
\end{minipage}
\end{figure}

If $i\in F_s$ and $j\in F_{s+1}$, then $i\prec j$ if and only if the point $(p_j-p_i,r_j-r_i)$ is among the three points in Figure \ref{relsP1}. For example, we see that for $i=1$, $i\prec j$ if and only if $(p_j-p_i,r_j-r_i)$ is among the points $(-1,1)$, $(-1,2)$, and $(-2,1)$. The only suitable values of $j$ are 5 and 10. The induced relations in $P(3,3,9)\smallsetminus \{0\}$ are depicted in Figure \ref{HasseP1}.

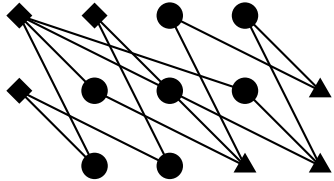
\begin{figure}[ht]
  \centering
\begin{tikzpicture}

    \node[diamond,fill=black,minimum size=10pt,inner sep=0pt] (10) at (0,0) {};
    \node[diamond,fill=black,minimum size=10pt,inner sep=0pt] [right of = 10] (11) {};
    \node[circle,fill=black,minimum size=10pt,inner sep=0pt] [right of = 11] (12) {};
    \node[circle,fill=black,minimum size=10pt,inner sep=0pt] [right of = 12] (13) {};
    
    \node[diamond,fill=black,minimum size=10pt,inner sep=0pt] [below of = 10] (5) {};
    \node[circle,fill=black,minimum size=10pt,inner sep=0pt] [below of = 11] (6) {};
    \node[circle,fill=black,minimum size=10pt,inner sep=0pt] [below of = 12] (7) {};
    \node[circle,fill=black,minimum size=10pt,inner sep=0pt] [below of = 13] (8) {};
    \node[regular polygon,regular polygon sides=3,fill=black,minimum size=10pt,inner sep=0pt] [right of = 8] (9) {};
    
    \node[regular polygon,regular polygon sides=3,fill=black,minimum size=10pt,inner sep=0pt] [below of = 9] (4) {};
    \node[regular polygon,regular polygon sides=3,fill=black,minimum size=10pt,inner sep=0pt] [below of = 8] (3) {};
    \node[circle,fill=black,minimum size=10pt,inner sep=0pt] [below of = 7] (2) {};
    \node[circle,fill=black,minimum size=10pt,inner sep=0pt] [below of = 6] (1) {};

	\draw [thick, shorten <=-2pt, shorten >=-2pt] (1) -- (5);
	\draw [thick, shorten <=-2pt, shorten >=-2pt] (1) -- (10);
	
	\draw [thick, shorten <=-2pt, shorten >=-2pt] (2) -- (5);
	\draw [thick, shorten <=-2pt, shorten >=-2pt] (2) -- (11);
	
	\draw [thick, shorten <=-2pt, shorten >=-2pt] (3) -- (6);
	\draw [thick, shorten <=-2pt, shorten >=-2pt] (3) -- (7);
	\draw [thick, shorten <=-2pt, shorten >=-2pt] (3) -- (12);
	
	\draw [thick, shorten <=-2pt, shorten >=-2pt] (4) -- (7);
	\draw [thick, shorten <=-2pt, shorten >=-2pt] (4) -- (8);
	\draw [thick, shorten <=-2pt, shorten >=-2pt] (4) -- (13);
	
	\draw [thick, shorten <=-2pt, shorten >=-2pt] (6) -- (10);
	
	\draw [thick, shorten <=-2pt, shorten >=-2pt] (7) -- (10);
	\draw [thick, shorten <=-2pt, shorten >=-2pt] (7) -- (11);
	
	\draw [thick, shorten <=-2pt, shorten >=-2pt] (8) -- (10);
	
	\draw [thick, shorten <=-2pt, shorten >=-2pt] (9) -- (12);
	\draw [thick, shorten <=-2pt, shorten >=-2pt] (9) -- (13);
\end{tikzpicture}
\caption{Constructing relations.}
\label{HasseP1}
\end{figure}

For $\ell=-2$ one can construct a similar but larger diagram as shown in Figure~\ref{relsP1}, which leads to the additional relations given in Figure~\ref{longRels}.
Combining these leads to our construction of the poset $P(3,3,9)$, depicted in Figure \ref{finalPoset}.

\begin{figure}[ht]
\centering
\begin{minipage}{.5\textwidth}
  \centering
\begin{tikzpicture}

    \node[diamond,fill=black,minimum size=10pt,inner sep=0pt] (10) at (0,0) {};
    \node[diamond,fill=black,minimum size=10pt,inner sep=0pt] [right of = 10] (11) {};
    \node[circle,fill=black,minimum size=10pt,inner sep=0pt] [right of = 11] (12) {};
    \node[circle,fill=black,minimum size=10pt,inner sep=0pt] [right of = 12] (13) {};
    
    \node[diamond,fill=black,minimum size=10pt,inner sep=0pt] [below of = 10] (5) {};
    \node[circle,fill=black,minimum size=10pt,inner sep=0pt] [below of = 11] (6) {};
    \node[circle,fill=black,minimum size=10pt,inner sep=0pt] [below of = 12] (7) {};
    \node[circle,fill=black,minimum size=10pt,inner sep=0pt] [below of = 13] (8) {};
    \node[regular polygon,regular polygon sides=3,fill=black,minimum size=10pt,inner sep=0pt] [right of = 8] (9) {};
    
    \node[regular polygon,regular polygon sides=3,fill=black,minimum size=10pt,inner sep=0pt] [below of = 9] (4) {};
    \node[regular polygon,regular polygon sides=3,fill=black,minimum size=10pt,inner sep=0pt] [below of = 8] (3) {};
    \node[circle,fill=black,minimum size=10pt,inner sep=0pt] [below of = 7] (2) {};
    \node[circle,fill=black,minimum size=10pt,inner sep=0pt] [below of = 6] (1) {};

	\draw [thick, shorten <=-2pt, shorten >=-2pt] (3) -- (5);
	\draw [thick, shorten <=-2pt, shorten >=-2pt] (3) -- (10);
	
	\draw [thick, shorten <=-2pt, shorten >=-2pt] (4) -- (5);
	\draw [thick, shorten <=-2pt, shorten >=-2pt] (4) -- (10);
	\draw [thick, shorten <=-2pt, shorten >=-2pt] (4) -- (11);
	
	\draw [thick, shorten <=-2pt, shorten >=-2pt] (9) -- (10);
	\draw [thick, shorten <=-2pt, shorten >=-2pt] (9) -- (11);
\end{tikzpicture}
\caption{Constructing additional relations.}
\label{longRels}
\end{minipage}%
\begin{minipage}{.5\textwidth}
  \centering
\begin{tikzpicture}

    \node[fill=white,minimum size=15pt,inner sep=0pt] (10) at (0,0) {10};
    \node[fill=white,minimum size=15pt,inner sep=0pt] [right of = 10] (11) {11};
    \node[fill=white,minimum size=15pt,inner sep=0pt] [right of = 11] (12) {12};
    \node[fill=white,minimum size=15pt,inner sep=0pt] [right of = 12] (13) {13};
    
    \node[fill=white,minimum size=15pt,inner sep=0pt] [below of = 10] (5) {5};
    \node[fill=white,minimum size=15pt,inner sep=0pt] [below of = 11] (6) {6};
    \node[fill=white,minimum size=15pt,inner sep=0pt] [below of = 12] (7) {7};
    \node[fill=white,minimum size=15pt,inner sep=0pt] [below of = 13] (8) {8};
    \node[fill=white,minimum size=15pt,inner sep=0pt] [right of = 8] (9) {9};
    
    \node[fill=white,minimum size=15pt,inner sep=0pt] [below of = 9] (4) {4};
    \node[fill=white,minimum size=15pt,inner sep=0pt] [below of = 8] (3) {3};
    \node[fill=white,minimum size=15pt,inner sep=0pt] [below of = 7] (2) {2};
    \node[fill=white,minimum size=15pt,inner sep=0pt] [below of = 6] (1) {1};
    \node[fill=white,minimum size=15pt,inner sep=0pt] at (5,-3) (0) {0};

	\draw [ shorten <=-2pt, shorten >=-2pt] (1) -- (5);
	\draw [ shorten <=-2pt, shorten >=-2pt] (1) -- (10);
	
	\draw [ shorten <=-2pt, shorten >=-2pt] (2) -- (5);
	\draw [ shorten <=-2pt, shorten >=-2pt] (2) -- (11);
	
	\draw [ shorten <=-2pt, shorten >=-2pt] (3) -- (5);
	\draw [ shorten <=-2pt, shorten >=-2pt] (3) -- (6);
	\draw [ shorten <=-2pt, shorten >=-2pt] (3) -- (7);
	\draw [ shorten <=-2pt, shorten >=-2pt] (3) -- (12);
	
	\draw [ shorten <=-2pt, shorten >=-2pt] (4) -- (5);
	\draw [ shorten <=-2pt, shorten >=-2pt] (4) -- (7);
	\draw [ shorten <=-2pt, shorten >=-2pt] (4) -- (8);
	\draw [ shorten <=-2pt, shorten >=-2pt] (4) -- (13);
	
	\draw [ shorten <=-2pt, shorten >=-2pt] (6) -- (10);
	
	\draw [ shorten <=-2pt, shorten >=-2pt] (7) -- (10);
	\draw [ shorten <=-2pt, shorten >=-2pt] (7) -- (11);
	
	\draw [ shorten <=-2pt, shorten >=-2pt] (8) -- (11);
	
	\draw [ shorten <=-2pt, shorten >=-2pt] (9) -- (10);
	\draw [ shorten <=-2pt, shorten >=-2pt] (9) -- (11);
	\draw [ shorten <=-2pt, shorten >=-2pt] (9) -- (12);
	\draw [ shorten <=-2pt, shorten >=-2pt] (9) -- (13);
	
	\draw [ shorten <=-2pt, shorten >=-2pt] (0) -- (1);
	\draw [ shorten <=-2pt, shorten >=-2pt] (0) -- (4);
	\draw [ shorten <=-2pt, shorten >=-2pt] (0) -- (2);
	\draw [ shorten <=-2pt, shorten >=-2pt] (0) -- (3);
	\draw [ shorten <=-2pt, shorten >=-2pt] (0) -- (9);
	
\end{tikzpicture}
\caption{The poset $P(3,3,9)$.}
\label{finalPoset}
\end{minipage}
\end{figure}

\end{example}


\section{Algebraic Implications}\label{sec:algebraic}

In this section, we discuss the algebraic implications of our analysis of the fundamental parallelepiped poset.
Our main result is Theorem~\ref{antichain} showing that the Poincar\'e series for the semigroup algebra associated to an antichain simplex is rational.
Unlike previous work of the authors~\cite{BraunDavisPoinc} establishing rationality of Poincar\'e series for lattice simplex semigroup algebras, our proof technique in this work involves the bar resolution.

\subsection{A Review of Resolutions and Poincar\'e Series}
For all background regarding graded resolutions of algebras, see Peeva's book~\cite{GradedSyzygies}.
Recall that the semigroup $(\Lambda,+)$ associated to a $d$-simplex $\Delta$ is the intersection $\Lambda\,:=\,\cone(\Delta)\cap\Z^{d+1}$ with $+$ given by the usual coordinate-wise addition on $\Z^{d+1}$.
The {\emph{semigroup algebra}} $K[\Lambda]$ associated to a semigroup $\Lambda\subset\Z^{d+1}$ is the $K$ vector space with basis $\{e_\alpha\}_{\alpha\in\Lambda}$ equipped with the product $e_\alpha\cdot e_\beta\,=\,e_{\alpha+\beta}$.
For $K$ a field, a $K$-algebra $R$ is called {\emph{graded}} with respect to $\Z^{n}$ if it can be written as a direct sum
\[R=\bigoplus_{\alpha\in\Z^n}R_\alpha
,\]
where for $x\in R_\alpha$ and $y\in R_\beta$, we have that $x\cdot y\in R_{\alpha+\beta}$. 
It is immediate that $K[\Lambda]$ is a $\Z^{d+1}$--graded $K$--algebra.
It is common to ``coarsen'' the grading of $K[\Lambda]$ by considering it to be a $\Z$-graded algebra with grading given by the zeroth coordinate of its $\Z^{d+1}$--grading.

In this context, the seemingly arbitrary definition of the cone over a simplex $\Delta$ is shown to be natural and helpful by the following observation.
For a point $x=(x_0,x_1,\dots,x_d)$ in $\R^{d+1}$, we define the {\emph{height}} of $x$ to be $\height(x)=x_0$.
Letting $X_n$ denote the collection of points $x\in\R^{d+1}$ with height equal to $n$, we have the set equality 
\[X_n\cap\cone(\Delta)\;=\;\{(n,n\cdot x)\in\R^{d+1}\text{ such that }x\in\Delta\}
  .\] Observe that the set $\Z^{d+1}\cap X_n\cap\cone(\Delta)$ is in bijection with the set of lattice points of $n\Delta$ (by dropping the zeroth coordinate).
Thus, the coarsened grading of $K[\Lambda]$ corresponds to the height function in the cone.

We need to consider complexes of vector spaces in order to define free resolutions of $K$-algebras.
Given a collection of vector spaces $\{F_i\}_{i\in\Z_{\geq0}}$, together with linear maps $\partial_i$ from $F_i$ to $F_{i-1}$, we call the sequence
\[F\,: \quad F_0\xleftarrow[]{\partial_1} F_1\xleftarrow[]{\partial_2}\cdots \xleftarrow[]{\partial_i}F_i \xleftarrow[]{\partial_{i+i}}F_{i+1}\xleftarrow[]{\partial_{i+2}}\cdots 
\]
a {\emph{complex}} of vector spaces if the image of $\partial_{i+1}$ is contained in the kernel of $\partial_i$ for all $i\geq1$.
The $i$'th homology of the complex $F$ is the quotient vector space $H_i(F)\,:=\,\ker\partial_{i}/\im\partial_{i+1}$.
Let $M$ be a finitely generated graded module over $R$, $F_i$ be a free $R$-module and $\partial_i$ be a graded $R$-module homomorphism such that the image of $\partial_{i+1}$ is equal to the kernel of $\partial_i$ for all $i\geq1$.
Then the complex $F$ is a {\emph{free resolution}} of $M$ over $R$ if $M\cong F_0/\im\partial_1$.
Because it is graded, we may split the free resolution $F$ into a direct sum of $K$ vector space complexes by writing each $F_i$ as a direct sum $\bigoplus_{\alpha\in\Z^{n}}F_{i,\alpha}$.

For $(F,\,\partial)$ a complex of free $R$-modules, we can define a tensor complex $(M\otimes F,\, \Id\otimes\partial)$. 
If $F$ is a graded free resolution of $M$, the Betti number $\beta_{i,\alpha}^R(M)$ of a graded $R$-module $M$ is the vector space dimension of the $i$'th homology of the graded component of $K\otimes F$ of degree $\alpha$.
This leads to our primary object of interest.

\begin{definition}The Poincar\'e series $P_R^M(z;\t)$ is the ordinary generating function for the Betti numbers of the $R$-module $M$, i.e.,
\[P_R^M(z;\t)\,=\,\sum_{\alpha\in\Z^n}\sum_{i\geq0}\beta_{i,\alpha}^R(M)z^i\t^\alpha
.\]
\end{definition}

In the case that $R$ is a polynomial ring in $n$ variables, the Hilbert Syzygy Theorem says that the Poincar\'e series $P_R^M(z;\t)$ is a polynomial for any finitely generated {$R$-module} $M$. However, when $R$ is not a polynomial ring, the growth of the Betti numbers is not so simple --- the Poincar\'e series may not even be rational.

\subsection{Rational Poincar\'e Series}
We call a $\Z^n$-graded algebra $R$ {\emph{connected}} if $R_0\cong K$ (as in the case of a semi-group ring $K[\Lambda]$ associated to a lattice simplex $\Delta$). By a slight abuse of notation, we write \[\m \,:=\,\bigoplus_{\alpha\in\Lambda\backslash0}R_\alpha \text{ and } K\,\cong\,R/\m\] as $R$-modules.
It has been shown \cite{gulliksen1970massey} that if the Poincar\'e series for the ground field $K$ as an $R$-module is rational for all $R$, then the Poincar\'e series is rational for any finitely generated module. Hence the question of Serre-Kaplansky:
\begin{question} Is the Poincar\'e series of the ground field $K$ over $R$ rational for all $K$-algebras $R$?
\end{question}
This question was answered in the negative by Anick~\cite{IrrPoincare}, and much subsequent work has focused on determining the properties of $R$ that lead to rationality or irrationality.
Our interest is in the rationality of the Poincar\'e series for $K[\Lambda]$, which leads us to define a related algebra as follows.

Because $K[\Lambda]$ is finitely generated (by its Hilbert basis $\HH$ given in Definition~\ref{def:hb}) it has a presentation 
\begin{equation}0\rightarrow \ker\varphi\rightarrow K[V_0,\dots,V_{d},x_1,\dots,x_m]\xrightarrow[]{\varphi}K[\Lambda]\rightarrow 0,
\end{equation}
where the map $\varphi$ is defined by the image of variables: the image of $V_i$ is the vector space basis element $e_{(1,v_i)}$ associated with the Hilbert basis element $(1,v_i)$ in $\Lambda$, and the image of $x_i$ is $e_{h_i}$ where the $h_i$ are the remaining elements of the Hilbert basis. 
This defines a surjective degree map $\deg(\cdot)$ from the set of monomials of $K[V_1,\dots,V_{d+1},x_1,\dots,x_m]$ onto $\Lambda$ by 
\[
\deg\left(\prod V_i^{s_i}\cdot\prod x_j^{r_j}\right)=\sum s_i(1,v_i) + \sum r_jh_j \, .
\]
Extending $\deg(\cdot)$ $K$-linearly, we see that $\ker\varphi$ is the toric ideal $I$ generated by all binomials 
\[
\mathbf{V}^{u_V}\mathbf{x}^{u_x}-\mathbf{V}^{w_V}\mathbf{x}^{w_x}
\]
such that $\deg\left(\mathbf{V}^{u_V}\mathbf{x}^{u_x})=\deg(\mathbf{V}^{w_V}\mathbf{x}^{w_x}\right)$. 

\begin{definition} The {\emph{Fundamental Parallelepiped Algebra}} $\fpa(\Delta)$ associated with the simplex $\Delta$ may be constructed in two ways; firstly as the quotient 
\[K[V_0,\dots,V_d,x_1,\dots,x_m]\,/\,\ker\varphi+\big(V_0,\dots,V_d\big),
\]
and secondly as the algebra with $K$ vector space basis \[\big\{e_{\sigma}\text{ such that }\sigma\in\Z^{d+1}\cap\Pi_{\Delta}\big\}\] and with multiplication given by \[e_\sigma\cdot e_\mu = \begin{cases} e_{\sigma+\mu} &\text{ if } \sigma+\mu\in\Z^{d+1}\cap\Pi_\Delta,\text{ and}\\
0 &\text{ otherwise.}
\end{cases}
\]
\end{definition}

One inspiration for defining this algebra is the fact that, due to an argument presented earlier, every element of $\Lambda$ may be written uniquely as a non-negative sum of points $(1,v_i)$ and a single point in $\Pi_\Delta$.
Because the generators $e_{(1,v_i)}$ form a linear system of parameters for $K[\Lambda]$, we have the following result which follows from~\cite[Prop. 3.3.5]{InfiniteFree}.
\begin{theorem}For the $\Z$-graded algebra $K[\Lambda]$, we have the following equality:\begin{align*}
P_{K[\Lambda]}^K(z;\t)&=\prod_i\left(1+z\t^{(1,v_i)}\right)\cdot P_{\fpa(\Delta)}^K(z;\t)\\
&= P_{K[V_0,\dots,V_d]}^K(z;\t)\cdot P_{\fpa(\Delta)}^K(z;\t).
\end{align*}
\end{theorem}

\subsection{Bar Resolutions and Antichain Simplices}

We will use the Bar resolution of $K$, with $K$ as a module over a graded $K$-algebra, which is a standard construction.
In the definition we use the bar symbol $|$ to mean a tensor over $K$, and reserve the tensor symbol $\otimes$ to mean a tensor over the ring under consideration.

\begin{definition} The Bar resolution $\B$ of the module $K$ over the $\Z^n$-graded $K$ algebra $\fpa(\Delta)$ has graded components $[\B_i]_\alpha$ with vector space basis given by $\delta_0\,|\,\cdots\,|\, \delta_i$ such that $\delta_0$ is in $\Pi_\Delta$, each $\delta_j$ is in $\Pi_{\Delta}\smallsetminus \{0\}$ (for $j\geq1$), and $\sum_{j=0}^i\delta_j=\alpha$. The differential map $\partial_i$ acts by sending $\delta_0\,|\,\cdots\,|\, \delta_i$ to the sum \[\sum_{j=0}^{i-1}(-1)^{j}\delta_0\,|\, \cdots\,|\, \delta_{j-1}\,|\, \delta_j+\delta_{j+1}\,|\, \delta_{j+2}\,|\, \cdots \,|\,\delta_i\] in $\B_{i-1}$.\end{definition}

Recall that in order to compute the Betti number $\beta_{i,\alpha}$ we must compute homology in the tensor complex $B:=\,K\otimes\B$. Because we identify $K$ with the vector sub-space $R_0$ with basis $e_0$, we see that $[B_i]_\alpha$ is generated as a vector space by the collection $\left\{e_0\otimes\delta_0\,|\,\delta_1\,|\,\cdots\,|\, \delta_i\right\}$. Observe that unless $\delta_0$ is equal the point 0 in $\Lambda$, the product $e_0\otimes\delta_0$ is equal to zero, since for $\sigma$ not equal to zero, $e_0\cdot e_\sigma$ is equal to zero in the module $K$, and hence
\[e_0\otimes e_\sigma\,=\,e_0\cdot e_\sigma\otimes e_0\,=\,0\otimes e_0\,=\,0
.\]
 Consequently, for $i\geq1$, $[B_i]_\alpha$ has a vector space basis in bijection with  the collection of $\delta_1\,|\,\cdots\,|\, \delta_i$ such that each $\delta_j$ is in $\Pi_{\Delta}\smallsetminus \{0\}$ and $\sum_{j=1}^i\delta_j=\alpha$. We further have that $[B_0]_\alpha$ is the trivial vector space unless $\alpha$ is zero in $\Lambda$, and that $[B_0]_{0}$ is isomorphic to $K$.

For a unimodular simplex $\Delta$, it is clear that the $\fpa(\Delta)$ is one-dimensional as a $K$ vector space, and has basis $e_0$. Consequently, $[B_i]_\alpha$ has empty basis (and dimension zero) unless $\alpha$ is equal to zero in $\Lambda$ and $i=0$.
It follows that the complex $B$ is given by
\[
0\leftarrow K\leftarrow0\leftarrow0\leftarrow\cdots
\]
and that
\[
\beta_{i,\alpha}=\begin{cases}1&\text{ if }i=0\text{ and }\alpha=0,\\
0&\text{ otherwise.}
\end{cases}
\]
Thus, $P_{\fpa(\Delta)}^K(z;\t)=1$. The result is consistent with the fact that $K[\Lambda]$ is a polynomial ring in the case that $\Delta$ is unimodular.

In the case of an antichain simplex, the differential map is uniformly zero, since $e_{\delta_j}\cdot e_{\delta_{j+1}}$ equals zero for all $j$. Further, $\beta_{i,\alpha}$ is equal to the dimension of $[B_i]_\alpha$. By considering the recurrence (for large $i$ and $\alpha$) 
\[\dim_K[B_i]_\alpha\,=\,\sum_{\substack{\sigma\in\PP(\Delta)\\\sigma\neq0}}\dim_K[B_i]_{\alpha-\sigma}
,\] we obtain the following.
\begin{theorem}\label{antichain} For an antichain simplex $\Delta$, we have
  \[P_{\fpa(\Delta)}^K(z;\t)=\;\left(\,1-\sum_{\substack{\sigma\in\PP(\Delta)\\\sigma\neq0}}z\t^\sigma\right)^{-1} \, ,
  \]
  and thus the Poincar\'e series is rational.
\end{theorem}

\bibliographystyle{amsplain}
\bibliography{BrianBib}

\addresseshere

\appendix
\newpage

\section{Experimental Data}\label{sec:appendix}

\begin{center}
  \begin{longtable}{|l|l|l|l|}
    \hline
  n & rpac($n$) & relprime($n$) & part($n$) \\
  \hline \hline
$1$ & $1$ & $1$ & $1$ \\
$2$ & $2$ & $2$ & $2$ \\
$3$ & $2$ & $2$ & $3$ \\
$4$ & $3$ & $4$ & $5$ \\
$5$ & $3$ & $3$ & $7$ \\
$6$ & $7$ & $10$ & $11$ \\
$7$ & $3$ & $3$ & $15$ \\
$8$ & $15$ & $21$ & $22$ \\
$9$ & $7$ & $8$ & $30$ \\
$10$ & $17$ & $22$ & $42$ \\
$11$ & $8$ & $8$ & $56$ \\
$12$ & $58$ & $76$ & $77$ \\
$13$ & $7$ & $7$ & $101$ \\
$14$ & $103$ & $134$ & $135$ \\
$15$ & $18$ & $21$ & $176$ \\
$16$ & $45$ & $56$ & $231$ \\
$17$ & $33$ & $38$ & $297$ \\
$18$ & $316$ & $384$ & $385$ \\
$19$ & $15$ & $16$ & $490$ \\
$20$ & $513$ & $626$ & $627$ \\
$21$ & $36$ & $41$ & $792$ \\
$22$ & $180$ & $215$ & $1002$ \\
$23$ & $78$ & $89$ & $1255$ \\
$24$ & $1317$ & $1574$ & $1575$ \\
$25$ & $31$ & $34$ & $1958$ \\
$26$ & $1169$ & $1414$ & $2436$ \\
$27$ & $148$ & $170$ & $3010$ \\
$28$ & $750$ & $874$ & $3718$ \\
$29$ & $143$ & $162$ & $4565$ \\
$30$ & $4779$ & $5603$ & $5604$ \\
$31$ & $26$ & $28$ & $6842$ \\
$32$ & $7050$ & $8348$ & $8349$ \\
$33$ & $392$ & $448$ & $10143$ \\
$34$ & $1675$ & $1951$ & $12310$ \\
$35$ & $478$ & $539$ & $14883$ \\
$36$ & $4850$ & $5625$ & $17977$ \\
$37$ & $115$ & $126$ & $21637$ \\
$38$ & $22109$ & $26014$ & $26015$ \\
$39$ & $816$ & $918$ & $31185$ \\
$40$ & $4410$ & $5047$ & $37338$ \\
$41$ & $433$ & $481$ & $44583$ \\
$42$ & $45819$ & $53173$ & $53174$ \\
$43$ & $104$ & $112$ & $63261$ \\
$44$ & $64731$ & $75174$ & $75175$ \\
$45$ & $1362$ & $1522$ & $89134$ \\
$46$ & $4192$ & $4747$ & $105558$ \\
$47$ & $2202$ & $2468$ & $124754$ \\
$48$ & $129242$ & $147272$ & $147273$ \\
$49$ & $365$ & $399$ & $173525$ \\
$50$ & $106948$ & $123165$ & $204226$ \\
$51$ & $1233$ & $1362$ & $239943$ \\
$52$ & $24641$ & $27874$ & $281589$ \\
$53$ & $3597$ & $3986$ & $329931$ \\
$54$ & $339300$ & $386154$ & $386155$ \\
$55$ & $623$ & $679$ & $451276$ \\
$56$ & $128590$ & $145176$ & $526823$ \\
$57$ & $3426$ & $3781$ & $614154$ \\
$58$ & $54230$ & $60927$ & $715220$ \\
$59$ & $8575$ & $9496$ & $831820$ \\
$60$ & $864231$ & $966466$ & $966467$ \\
$61$ & $302$ & $324$ & $1121505$ \\
$62$ & $1146930$ & $1300155$ & $1300156$ \\
$63$ & $13151$ & $14458$ & $1505499$ \\
$64$ & $55541$ & $61850$ & $1741630$ \\
$65$ & $16496$ & $18200$ & $2012558$ \\
$66$ & $522255$ & $586074$ & $2323520$ \\
$67$ & $1012$ & $1091$ & $2679689$ \\
$68$ & $2761384$ & $3087734$ & $3087735$ \\
$69$ & $20580$ & $22503$ & $3554345$ \\
$70$ & $234794$ & $261034$ & $4087968$ \\
$71$ & $3040$ & $3287$ & $4697205$ \\
$72$ & $4875893$ & $5392782$ & $5392783$ \\
$73$ & $2715$ & $2931$ & $6185689$ \\
\hline
\end{longtable}
\end{center}

\end{document}